\documentclass[11pt]{article}

\usepackage{amsfonts,amsmath,amsthm,latexsym,color,epsfig,mathrsfs}
\newtheorem{thm}{Theorem}[section]

\newtheorem{lem}[thm]{Lemma}
\newtheorem{cor}[thm]{Corollary}
\newtheorem{prop}[thm]{Proposition}
\newtheorem{conjecture}[thm]{Conjecture}
\newtheorem{question}[thm]{Question}

\newcommand{\A}{\mathcal{A}}

\title{Long geodesics in subgraphs of the cube}
\author{
Imre Leader\thanks{Department of Pure Mathematics and Mathematical Statistics, Centre for Mathematical Sciences, 
Cambridge CB3 0WB, United Kingdom. E-mail: I.Leader@dpmms.cam.ac.uk} \and 
Eoin Long\thanks{School of Mathematical Sciences, Queen Mary University of London, Mile End Road, London E1 4NS, United Kingdom. E-mail: E.P.Long@qmul.ac.uk}
}

\begin{document}

%
% Title
%
\maketitle

%
% Abstract
%

\begin{abstract}
A path in the hypercube $Q_n$ is said to be a geodesic if no two of its edges are in 
the same direction. Let $G$ be a subgraph of $Q_n$ with average degree $d$. How long a 
geodesic must $G$ contain? We show that $G$ must contain a geodesic of length $d$. This 
result, which is best possible, strengthens a theorem of Feder and Subi. It is also 
related to the `antipodal colourings' conjecture of Norine.
\end{abstract}

%
%
% First Section:
%
% Introduction
%
%
%

\section{Introduction}

Given a graph $G$ of average degree $d$, a classic result of Dirac \cite{dirac} guarantees 
that $G$ contains a path of length $d$. Moreover, for general graphs this is the best 
possible bound, as can be seen by taking $G$ to be $K_{d+1}$, the complete graph on $d+1$ 
vertices. 

The hypercube $Q_n$ has vertex set $\{0,1\}^n$ and two vertices $x,y\in Q_n$ are joined by 
an edge if they differ on a single coordinate. In \cite{long} the second author considered 
a similar question for subgraphs of the hypercube $Q_n$. That is, given a subgraph $G$ of 
$Q_n$ of average degree $d$, how long a path must $G$ contain? The main result 
was the following:

\begin{thm}[\cite{long}]
\label{previous}
Every subgraph $G$ of $Q_n$ of minimum degree $d$ contains a path of length $2^d-1$.
\end{thm}

\noindent Combining Theorem \ref{previous} with the standard fact that any graph of average 
degree $d$ contains a subgraph with minimum degree at least $d/2$, we see that any subgraph 
$G$ of $Q_n$ with average degree $d$ contains a path of length at least $2^{d/2}-1$.

In this paper we consider the analogous question for geodesics. A path in $Q_n$ is a geodesic 
if no two of its edges have the same direction. Equivalently, a path is a geodesic if it forms 
a shortest path in $Q_n$ between its endpoints. Given a subgraph $G$ of $Q_n$ of average degree 
$d$, how long a geodesic path must $G$ contain? 

It is trivial to see that any such graph must contain a geodesic of length $d/2$. Indeed, taking 
a subgraph $G'$ of $G$ with minimal degree at least $d/2$ and starting from any vertex of $G'$, 
we can greedily pick a geodesic of length $d/2$ by choosing a new edge direction at each step.

On the other hand the $d$-dimensional cube $Q_d$ shows that, in general, we cannot find a 
geodesic of length greater than $d$ in $G$. Our main result is that this upper bound is 
sharp.

%
% Main result
%

\begin{thm}
\label{mainresult}
Every subgraph $G$ of $Q_n$ of average degree $d$ contains a geodesic of length at least $d$.
\end{thm}

Noting that the endpoints of the geodesic in $G$ guaranteed by Theorem \ref{mainresult} 
are at Hamming distance at least $d$, we see that Theorem \ref{mainresult} extends the following 
result of Feder and Subi \cite{fedsubi}. 

%who proved that there exists \textit {some} path between two vertices at Hamming distance 
%$d$ apart.

%
% Feder Subi
%

\begin{thm}[\cite{fedsubi}]
\label{FS}
{
Every subgraph $G$ of $Q_n$ of average degree $d$ contains two vertices at Hamming 
distance $d$ apart.
}
\end{thm}

We remark that neither Theorem \ref{mainresult} nor Theorem \ref{FS} follow from isoperimetric 
considerations alone. Indeed, if $G$ is a subgraph of $Q_n$ of average degree $d$, by the edge 
isoperimetric inequality for the cube (\cite{bernstein}, \cite{harper}, \cite{hart}, \cite{lindsey}; 
see \cite{com} for background) we have $|G|\geq 2^d$. However if $n$ is large, a Hamming ball of 
small radius may have size larger than $2^d$ without containing a long geodesic.

While Theorem \ref{mainresult} implies Theorem \ref{FS}, we have also given an alternate 
proof of Theorem \ref{FS} from a result of Katona \cite{ka} which we feel may be of interest. 
Both of these proofs will be given in the next section.

Finally, Feder and Subi's theorem was motivated by a conjecture of Norine \cite{norine} on antipodal 
colourings of the cube. In the last section of this short paper we discuss Theorem 
\ref{mainresult} in relation to Norine's conjecture.

%
%
%
%
%
%
% Second Section:
% 
% Proofs
%
%
%
%

\section{Proofs of Theorem \ref{mainresult} and Theorem \ref{FS}}

To prove Theorem \ref{mainresult} we will actually establish a stronger result. Before stating this 
result we need the following definition.

\begin{definition}
A path $P=x_1x_2\ldots x_l$ in $Q_n$ is an \textit{increasing geodesic} if the directions of the edges $x_ix_{i+1}$ increase with $i$. An increasing geodesic $P$ \textit{ends} at a vertex $x$ if $x=x_l$.\end{definition}

Since any increasing geodesic is also a geodesic, to prove Theorem \ref{mainresult} it is enough to show that any subgraph of $Q_n$ of average degree $d$ contains an increasing geodesic of length $d$. 

In fact for our proof we need to show more than this. For any vertex $x\in G$ we let $L_G(x)$ denote an increasing geodesic in $G$ of maximal length which ends at $x$. The key idea to the proof is to show that on average $|L_G(x)|$ is large. This allows us to simultaneously keep track of geodesics for all vertices of $G$, which is vital in the inductive proof below.

\begin{thm}
\label{averagemonotone}
Let $G$ be a subgraph of $Q_n$ of average degree $d$. Then
\begin{equation*}
 \sum _{v\in V(G)} |L_G(v)| \geq d|G|.
\end{equation*}
\end{thm}

Note that it is immediate from Theorem \ref{averagemonotone} that $|L_G(v)|\geq d$ for some 
$v\in V(G)$ and therefore $G$ contains an increasing geodesic of length at least $d$, as claimed.

\begin{proof}
{
Write $S(G)$ for $\sum _{v\in V(G)} |L_G(v)|$. We will show that for any subgraph $G$ of $Q_n$ 
we have $S(G) \geq 2|E(G)|$, by induction on $|E(G)|$. The base case $|E(G)|=0$ is immediate. Assume the result 
holds by induction for all graphs with $|E(G)|-1$ edges and that we wish to prove the result for $G$.

Pick an edge $e=xy$ of $G$ with largest coordinate direction and look at the graph $G'=G-e$. By the induction 
hypothesis we have 
\begin{equation*}
S(G')=\sum _{v\in V(G')} |L_{G'}(v)| \geq 2|E(G')| = 2(|E(G)|-1).
\end{equation*}
Now clearly we must have $|L_G(v)|\geq |L_{G'}(v)|$ for all vertices $v\in G$. Furthermore, notice that the 
coordinate direction of $e$ can not appear on the increasing geodesics $L_{G'}(x)$ and $L_{G'}(y)$. 
Indeed, the edge of $L_{G'}(x)$ adjacent to $x$ has direction less than $e$ and as $L_{G'}(x)$ is 
an increasing geodesic, the directions of all edges in $L_{G'}(x)$ must be less than $e$. We now consider two cases:

%
% Case (i)

\textbf{Case I:} $|L_{G'}(x)|=|L_{G'}(y)|$. Then the paths $L_{G'}(x)xy$ and $L_{G'}(y)yx$ are increasing geodesics in $G$ 
ending at $y$ and $x$ respectively. Therefore $|L_{G}(x)|\geq |L_{G'}(x)|+1$ and $|L_{G}(y)| \geq |L_{G'}(y)|+1$ 
and $S(G)\geq S(G')+2\geq 2|E(G')|+2=2|E(G)|$.

%
% Case (ii)

\textbf{Case II:} $|L_{G'}(x)|\neq |L_{G'}(y)|$. Without loss of generality assume that $|L_{G'}(x)|\geq |L_{G'}(y)|+1$. Then $L_{G'}(x)xy$ is an increasing geodesic ending at $y$ of length $|L_{G'}(x)|+1\geq |L_{G'}(y)| + 2$. Therefore $|L_{G}(y)|\geq |L_{G'}(y)|+2$ and $S(G)\geq S(G')+2\geq 2|E(G')|+2=2|E(G)|$. 

This concludes the inductive step and the proof.
}
\end{proof}

%
%
%
%
%
% End of proof of main result
%
%
%
%
%
%

We now give a strengthening of Theorem \ref{averagemonotone}, showing that $G$ must actually contain \emph{many} 
geodesic of length $d$. First note that for $d\in {\mathbb {N}}$, taking a disjoint 
union of subgraphs isomorphic to $Q_d$ gives a graph $G$ with average degree $d$ and exactly 
$d!|G|/2$ geodesics of length $d$. The following result shows shows that in fact we can guarantee that many 
geodesics of length $d$ for general subgraphs of $Q_n$.

\begin{thm}
 \label{numberofgeodesics}
 If $G$ is a subgraph of $Q_n$ with average degree at least $d\in \mathbb{N}$, then $G$ contains at least 
 $\frac{d!|G|}{2}$ geodesics of length $d$.
\end{thm}

\begin{proof}
 We first use Theorem \ref{averagemonotone} to prove the following claim: $G$ 
 contains at least $|G|$ increasing geodesics of length $d$. To see this, first remove an edge 
 $e$ from $G$ if it lies in at least two increasing geodesics of length $d$. 
 Now repeat this with $G\setminus \{e\}$ and so on until we end up at a subgraph $G'$ of $G$ in which all 
 edges lie in at most one increasing geodesic of length $d$. Let $e(G) = e(G') + a$. Note that, 
 by our removal process, the $a$ edges removed from $G$ remove at least $2a$ increasing geodesics 
 of length $d$. Therefore if $a \geq |G|/2$, then $G$ contains at least $|G|$ increasing geodesics of 
 length $d$. If not, by Theorem \ref{averagemonotone} we have 
\begin{equation}
 \label{countingmonotoneincreasingpathsinG}
\sum _{v\in G'} |L_{G'}(v)| \geq 2 e(G') =  2e(G) - 2a \geq d|G| - 2a = (d-1)|G| + (|G| - 2a).
\end{equation}
 Now note that since no edge of $G'$ is contained in more than one increasing geodesic of length $d$, $G'$ does 
 not contain any increasing geodesics of length  $d+1$. Therefore $|L_{G'}(v)| \leq d$ for all $v\in G'$. 
 By (\ref{countingmonotoneincreasingpathsinG}) this shows that $|L_{G'}(v)| = d$ for at least $|G|-2a$ 
 vertices $v\in G'$. Combining these with the increasing geodesics of length $d$ containing edges 
 from $G\setminus G'$, this shows that $G$ contains at least $2a + (|G| - 2a) = |G|$ increasing 
 geodesics of length $d$, as claimed.

 Now suppose that $G$ contain $L$ geodesics of length $d$. We will show that $L\geq \frac{d!|G|}{2}$. To see this, 
 pick an ordering $\sigma $ of $\{1,\ldots ,n\}$ uniformly at random and consider the geodesics 
 of length $d$ which are increasing with respect to this ordering (i.e. paths in which the 
 edges have directions $\sigma (i_1), \sigma (i_2),\ldots ,\sigma (i_d)$ where $i_1<i_2<\ldots <i_d$). 
 The probability that a fixed geodesic of length $d$ appears as an increasing geodesic with 
 respect to the ordering $\sigma $ is exactly $\frac{2}{d!}$. Taking $X$ to be the random variable 
 which counts the number of increasing geodesics of length $d$ in $G$ (with respect to 
 the ordering $\sigma$), this gives that 
 \begin{equation*}
  \mathbb{E}(X) = \frac{2L}{d!}.
 \end{equation*}
 But by the claim above, $X \geq |G|$ for each choice of $\sigma $. Therefore $L \geq \frac{d!|G|}{2}$, as 
 required.
\end{proof}

%
%
% Proof of Feder-Subi result 
% using the Katona's theorem
%
%

We now give the alternate proof of Theorem \ref{FS}. Note that it is enough to prove 
this theorem for induced subgraphs of $Q_n$, since if the result fails for some graph $G$, it must 
also fail for the induced subgraph of $Q_n$ on vertex set $V(G)$. 

As in \cite{fedsubi}, the following compression operation allows us a further reduction. 
Here we view the vertices of $Q_n$ as elements of $\mathcal{P}[n]$, the power set of $[n]$. 
Given $A\in \mathcal{P}[n]$ and $i\in \{1,\ldots ,n\}$ we let 
\begin{equation*}
{
  C_i(A) = \Big \{ \begin{array}{ll}
         A-i & \mbox{if $i\in A$};\\
        A & \mbox{if $i\notin A$.}\end{array} 
}
\end{equation*}

Given $\mathcal{A}\subset \mathcal{P}[n]$, $C_i(\mathcal{A}):=\{ C_i(A):A\in \mathcal{A}\}\cup \{A:C_i(A)\in \mathcal{A}\}$, 
the down compression of $\mathcal{A}$ in the $i$-direction. A family $\mathcal{A}$ is said to be a downset if 
$C_i(\mathcal{A})=\mathcal{A}$ for all $i\in [n]$. The following lemma shows that we may also assume that 
the vertex set $V(G)$ is a downset.

%
%
%
%
% Compression lemma in Katona proof 
% of Feder Subi
%
%
%

\begin{lem} 
\label{compression}
Let $G$ be an induced subgraph of $Q_n$ on vertex set $\mathcal{A}\subset \mathcal{P}[n]$ and let $i\in \{1,\ldots,n\}$. 
Suppose $G$ has average degree at least $d$ and all vertices $A$ and $B$ of $G$ are at Hamming distance less than $k$. 
Then the same is true for the induced subgraph $G'$ of $Q_n$ with vertex set $C_{i}(\mathcal{A})$.
\end{lem}

%
%
%
% Proof of compression lemma
%
%
%

\begin{proof}
{
Since $|G|=|G'|$ in both cases, to see that $G'$ has average degree at least $d$ it suffices to show that 
$G'$ has at least as many edges as $G$. To see this, define a map $f:E(G)\rightarrow E(G')$ given by
\[f(AB) = \left\{ \begin{array}{ll}
         C_i(A)C_i(B) & \mbox{if $A\Delta B\neq \{i\}$ and $C_i(A)C_i(B)\notin E(G)$};\\
        AB & \mbox{otherwise.}\end{array} \right. \]
Noting that $f$ is an injection, it follows that $G'$ has average degree at least $d$.

Suppose for contradiction that $G'$ had two vertices $A'$ and $B'$ at Hamming distance at least $k$ apart. 
Now it is easily seen that exactly one of $A'$ and $B'$ must contain $i$ as otherwise any pair $A,B\in \A $
with $C_i(A)=A'$ and $C_i(B)=B'$ are at Hamming distance at least $k$ apart. Assume that $i\in A'$, $i\notin B'$. 
Now $A'\in C_i(\mathcal{A})$ implies that $A'-i,A'\in \mathcal{A}$. Since $A'\in \A$, $B'\notin \A $ and 
we have $B'\in C_i(\A )\backslash \A $. This implies $B'\cup \{i\}\in \A $. But then $A'-i, B'\cup \{i\}\in \A $ 
are at Hamming distance at least $k$, a contradiction.
}
\end{proof}

%
%
%
% Defining shadows and $t$-interesecting families
%
%
%
%

As mentioned in the Introduction, our alternate proof of Theorem \ref{FS} is based on a theorem of Katona. 
Before stating this theorem we first need a definition.

\begin{definition}
{
Given a set system $\mathcal{A}\subset {[n]}^{(k)}$, the {\textit {shadow}} of $\mathcal{A}$ is 

\begin{equation*}
{
\partial{(\mathcal{A}) } :=\{ B\in {[n]}^{(k-1)}: B\subset A \mbox{ for some } A\in \A \}
}
\end{equation*}

The set $\partial ^{(l)}(\A)$ is defined as $\partial ^{(l)}(\A):=\overbrace{ \partial(\cdots ({\partial} (\A))\cdots )}^l$.
}
\end{definition}
\vspace{2mm}

While in general the shadow $\partial {\mathcal A}$ of $\mathcal A\subset \mathcal{P}{[n]}$ can be much smaller than $|\mathcal A|$, 
a result of Katona \cite{ka} shows that if $\mathcal A$ is also an intersecting family then $|\partial {\mathcal A}|\geq |{\mathcal A}|$. 
More generally, Katona also gave lower bounds on the size of $|\partial ^{(l)}{(\mathcal{\A })}|$ for $t$-intersecting families $\mathcal A$. 
We will need the following special case.

%
%
%
%
% Statement of Katona's theorem
%
%
%
%

\begin{thm}[Katona]
{
\label{Katona}
Let $k, t\in {\mathbb {N}}$. Suppose that $\A\subset {[n]}^{(k)}$ is $t$-intersecting. Then
\begin{equation*}
{
|\partial ^{(t)}({\mathcal{A}})|\geq |{\mathcal {A}}|
}
\end{equation*}
}
\end{thm}
\vspace{2mm}

%
%
%
%
% Proof of Feder Subi result 
% from Katona's theorem
%
%
%
%
%

\textit {Proof of Theorem \ref{FS}.} 
Suppose for contradiction the result is false and let ${\mathcal A}$ be the vertex set of $G$. 
Using Lemma \ref{compression} we may assume that $\mathcal{A}$ is down-compressed. 

Let ${\A} ^{(k)}=\A \cap {[n]}^{(k)}$ for all $k\in [n]$. Since $\mathcal{A}$ is down-compressed we must 
have $\A ^{(k)}=\emptyset $ for all $k\geq d$. Also since $\mathcal{A}$ is down-compressed, for each 
$A\in \mathcal{A}$, the number of neighbours of $A$ which lie below $A$ in $G$ is $|A|$. Therefore

\begin{equation}
{
\label{equality}
\sum_{k=0}^{\lceil d\rceil -1}k|\mathcal {A} ^{(k)}| = \sum_{A\in \A}|A| =  \frac{d|\A|}{2}.
}
\end{equation}

Furthermore, again by compression, for $k\geq \frac{d}{2}$, $\A ^{(k)}$ does not contain two vertices 
$A$ and $B$ with $|A\cup B|\geq d$. Therefore ${\mathcal A} ^{(k)}$ must be $(2k-\lceil d\rceil +1)$-intersecting. 
Applying Theorem \ref{Katona} we therefore have

\begin{equation}
\label{degreesum}
{
|\partial ^{(2k-\lceil d\rceil +1)}(\mathcal{A}^{(k)}|\geq |\A ^{(k)}|.
}
\end{equation}

But as $\mathcal {A}$ is down-compressed 
\begin{equation*}
 \partial ^{(2k-\lceil d\rceil +1)}(\A ^{(k)}) \subset \A ^{(\lceil d\rceil -k -1)}.
\end{equation*}
We now pair the contributions from ${\A} ^{(k)}$ and $\A ^{(\lceil d\rceil -k-1)}$ to (\ref{equality})
together for all $k\geq (\lceil d\rceil -1)/2$ using (\ref{degreesum}):

\begin{equation*}
\begin{split}
k|\A ^{(k)}| + (\lceil d\rceil -k -1)|\A ^{(\lceil d\rceil -k -1)}|
& =
(\lceil d\rceil - 1)/2|\A ^{(k)}| + (k- (\lceil d\rceil -1)/2)|\A ^{(k)}|\\ 
& + (\lceil d\rceil -1)/2|\A ^{(\lceil d\rceil -k-1)}| + ((\lceil d\rceil -1)/2 -k)|\A ^{(\lceil d\rceil -k-1)}|\\
& \leq (\lceil d\rceil -1) /2(|\A ^{(k)}|+|\A ^{(\lceil d\rceil -k-1)}|).
\end{split}
\end{equation*}

%
%
% Last line in the proof of FS
%
%

But summing over $k\geq (\lceil d\rceil -1)/2$ this contradicts (\ref{equality}) above. This proves the theorem. 

\hspace {15cm}\qquad \qquad $\square $

%
%
%
%
%
%
%
%
%
% Final section:
%
% Concluding Remarks
%
%
%
%
%
%
%
%

\section{Concluding Remarks}

We now discuss the relation of Theorem \ref{mainresult} with Norine's conjecture (see \cite{norine}) mentioned 
in the Introduction. Given a vertex $x\in Q_n$, its antipodal vertex $x' \in Q_n$ is the unique vertex with all 
coordinate entries differing from those of $x$. Also, given an edge $e=xy$ of $Q_n$, its antipodal edge $e'=x'y'$ where 
$x'$ is antipodal to $x$ and $y'$ is antipodal to $y$. Finally, a $2$-colouring of the edges of $Q_n$ is said to be 
\emph{antipodal} if no two antipodal edges receive the same colour.

\begin{conjecture}[Norine]
{
For $n\geq 2$, any antipodal colouring of $E(Q_n)$ contains a monochromatic path between two antipodal points. 
}
\end{conjecture}

Note that this is not true for general $2$-colourings of $E(Q_n)$, as can be seen by colouring all edges in directions 
$\{1,\ldots n-1\}$ red and edges in direction $n$ blue. In \cite{fedsubi}, Feder and Subi made the following conjecture 
for general $2$-colourings of $E(Q_n)$:

\begin{conjecture}[Feder-Subi]
\label{generalcolouring}
{
For every $2$-colouring of $E(Q_n)$ there exists a path between some pair of antipodal vertices which changes colour 
at most once.
}
\end{conjecture}

It is easily seen that if Conjecture \ref{generalcolouring} is true, it implies Norine's conjecture. Indeed, 
given an antipodal colouring of $Q_n$ take the path $P$ guaranteed by Conjecture \ref{generalcolouring} between 
two antipodal vertices in $Q_n$. Combining $P$ with its antipodal path $P^A$ then gives that some two antipodal 
vertices on $PP^A$ must be joined by a monochromatic path.

In \cite{fedsubi} Feder and Subi proved that every $2$-colouring of $E(Q_n)$ contains a monochromatic path between two 
vertices at Hamming distance $\lceil n/2\rceil $. Using Theorem \ref{mainresult} in place of Theorem \ref{FS} the following 
shows that we can actually take this path to be a geodesic.

\begin{cor}
Every $2$-colouring $c$ of $E(Q_n)$ contains a monochromatic geodesic of length $\lceil n/2\rceil $. 
\end{cor}

\begin{proof}
Pick a monochromatic connected component $C$ of the colouring with average degree at least $n/2$ and apply Theorem \ref{mainresult} 
to it.
\end{proof}

This suggests that in both of the conjectures above, one can additionally ask for the path between antipodal vertices to be a 
geodesic.

\begin{conjecture} The following statements hold:

\begin{itemize} 
\item[\textbf{A}] {Every antipodal colouring $c$ of $E(Q_n)$ contains a monochromatic geodesic between some pair of antipodal vertices.}
\item[\textbf{B}] {In every $2$-colouring $c$ of $E(Q_n)$, there is a geodesic between antipodal points which changes colour at most once.} 
\end{itemize}
\end{conjecture}

Unfortunately we were not able to settle either of these conjectures. In fact, surprisingly, we were not even able to 
establish that in every $2$-colouring of $E(Q_n)$ some two antipodal vertices are joined by a path which changes 
colour $o(n)$ times. Is this true? 

\begin{question} 
{
Is it true that for every $2$-colouring of $E(Q_n)$, there exist two antipodal vertices 
$x$ and $x'$ that are joined by a path that changes colour $o(n)$ times?
}
\end{question}

While we were not able to prove either \textbf{A} and \textbf{B}, our final result shows that are equivalent.

%
%
%
% Equivalence of geodesic conjectures
%
%
%
\begin{prop}
\textbf{A} holds for all $n$ if and only if \textbf{B} holds for all $n$.
\end{prop}

\begin{proof}
{
First assume that \textbf{A} is true and let $c$ be a $2$-colouring of $E(Q_n)$. View $Q_n$ as the subcube of $Q_{n+1}$ 
consisting of all $0-1$ vectors of length $n+1$, $(x_1,x_2,\ldots ,x_{n+1})$ with $x_{n+1}=0$. Pick any antipodal colouring 
$c'$ of $E(Q_{n+1})$ which agrees with $c$ on $E(Q_n)$. \textbf{A} now guarantees $c'$ has a monochromatic geodesic  
$P$ between two antipodal vertices of $Q_{n+1}$. Let $P^A$ denote the geodesic formed by the edges antipodal to $P$. Since 
$c'$ is antipodal, $P^A$ must also be monochromatic and of opposite colour to $P$. The restriction of the cycle $PP^A$ to our 
original subcube $Q_n$ now gives a geodesic between two antipodal vertices (in $Q_n$) which changes colour at most once, 
i.e. \textbf{B} is true.

Now assume that \textbf{B} is true and let $c$ be an antipodal $2$-colouring of $E(Q_n)$. Applying \textbf{B} 
to $c$ we obtain a geodesic $P$ between two antipodal vertices which changes colour at most once. Let $P=P_rP_b$ where 
$P_r$ is a red geodesic and $P_b$ is a blue geodesic. But since $c$ is antipodal $P_r^A$ is a blue geodesic 
and $P_bP_r^A$ is a blue geodesic between antipodal vertices, i.e. \textbf{A} is true.
} 
\end{proof}


\begin{thebibliography}{99}
\bibitem{bernstein} A.J. Bernstein: Maximally connected arrays on the $n$-cube,
\textit{SIAM J. Appl. Math.} \textbf{15}(1967), 1485-1489.
%
\bibitem{com} B. Bollob\'as: \textit{Combinatorics: Set Systems, Hypergraphs,
Families of Vectors and Combinatorial Probability}, Cambridge University Press,
1st ed, 1986.
%
\bibitem{dirac} G.A. Dirac: Some theorems on abstract graphs, 
\textit{Proc. London Math. Soc.} \textbf{2} (1952), 69-81.
%
\bibitem{fedsubi} T. Feder, C. Subi: On hypercube labellings and antipodal monochromatic paths, 
\textit{http://theory.stanford.edu/$\sim$tomas/antipod.pdf}.
%
\bibitem{harper} L.H. Harper: Optimal assignments of numbers to vertices,
\textit{SIAM J. Appl. Math.} \textbf{12}(1964), 131-135.
%
\bibitem{hart} S. Hart: A note on edges of the $n$-cube, \textit{Discrete Math.}
\textbf{14}(1976), 157-163.
%
\bibitem{ka} G.O.H Katona: Intersection theorems for systems of finite sets. 
\emph{Acta Math. Acad. Sci. Hung.} \textbf{15}, (1964), 329-337.
%
\bibitem{lindsey} J.H. Lindsey: Assignment of numbers to vertices, \textit{Amer.
Math. Monthly} \textbf{71}(1964), 508-516.
%
\bibitem{long} E. Long: Long paths and cycles in subgraphs of the cube, 
\textit{Combinatorica, to appear}.
%
\bibitem{norine} S. Norine, Open Problem Garden, \newline 
\textit{http://garden.irmacs.sfu.ca/?q$=$op/edge\_antipodal\_colorings\_of\_cubes}.
\end{thebibliography}
\end{document}